\documentclass[journal,twoside,web]{ieeecolor}

\usepackage{generic}
\usepackage{cite}
\usepackage{amsmath,amssymb,amsfonts}
\usepackage{algorithmic}
\usepackage{graphicx}
\usepackage{algorithm,algorithmic}
\usepackage{hyperref}
\hypersetup{hidelinks=true}
\usepackage{textcomp}

\usepackage{tikz}		
\usepackage{fontawesome}  
\usepackage{pgfplots}
\pgfplotsset{compat=1.12}
\usepgfplotslibrary{fillbetween}
\usetikzlibrary{automata, positioning, arrows, patterns, decorations}
\usepackage{float}

\def\JS{\textcolor{blue}}
\def\JS{\textcolor{black}}

\definecolor{darkgreen}{rgb}{0,0.6,0}

\def\eps{\varepsilon}

\newcommand{\normX}[1]{\JS{\lVert}#1\JS{\rVert}}

\newcommand{\args}{(\cdot)}
\newcommand{\Xtraj}[1]{\JS{X_{\mathbf{U}}(#1,X_0)}}
\newcommand{\XtrajShort}[1]{\JS{X(#1)}}
\newcommand{\Xstraj}[1]{\JS{X^s(#1)}}
\newcommand{\XstrajShort}[1]{\JS{X^s(#1)}}
\newcommand{\R}{\mathbb{R}}
\newcommand{\N}{\mathbb{N}}

\newcommand{\K}{\mathcal{K}}
\newcommand{\U}{\mathcal{U}}
\newcommand{\X}{\mathcal{X}}
\newcommand{\W}{\mathcal{W}}
\newcommand{\F}{\mathcal{F}}
\newcommand{\Y}{\mathbb{Y}}
\newcommand{\XX}{\mathbb{X}}
\newcommand{\UU}{\mathbb{U}}
\newcommand{\FF}{\mathbb{F}}
\newcommand{\St}{\mathbb{S}}
\newcommand{\PP}{\mathcal{P}}

\newcommand{\Lp}[2][p]{L^{#1}(\Omega,\F_k,\mathbb{P};#2)}
\newcommand{\Llp}[2][p]{L^{#1}(\Omega,\F,\mathbb{P};#2)}

\newcommand{\Pset}{\mathcal{S}}
\newcommand{\Dset}{\mathcal{D}}
\newcommand{\Mset}{\mathcal{M}}
\newcommand{\Lset}{\mathcal{L}}
\newcommand{\Normal}{\mathcal{N}}

\newcommand{\Prob}{\mathbb{P}}
\newcommand{\E}{\mathbb{E}}

\newcommand{\Exp}[1]{\mathbb{E}\left[#1\right]}

\newtheorem{thm}{Theorem}[section]
\newtheorem{prop}[thm]{Proposition}
\newtheorem{lem}[thm]{Lemma}
\newtheorem{bsp}[thm]{Example}
\newtheorem{defn}[thm]{Definition}

\newtheorem{rem}[thm]{Remark}

\def\BibTeX{{\rm B\kern-.05em{\sc i\kern-.025em b}\kern-.08em
    T\kern-.1667em\lower.7ex\hbox{E}\kern-.125emX}}

\begin{document}
\title{On the relationship between stochastic turnpike and dissipativity notions}
\author{Jonas Schießl, Michael H. Baumann, Timm Faulwasser, \IEEEmembership{Senior Member, IEEE}, and Lars Grüne
\thanks{This work was funded by the Deutsche Forschungsgemeinschaft (DFG, German Research Foundation) \-- project number 499435839. }
\thanks{Jonas Schießl, Michael H. Baumann, and Lars Grüne are with Mathematical Institute, University of Bayreuth, Germany (e-mail: jonas.schiessl@uni-bayreuth.de; michael.baumann@uni-bayreuth.de; lars.gruene@uni-bayreuth.de). }
\thanks{Timm Faulwasser was with Institute for Energy Systems, Energy Efficiency and Energy Economics, TU Dortmund University, Germany when major parts of this work were done. Current address: Institute of Control Systems, Hamburg University of Technology, Hamburg, Germany (e-mail: timm.faulwasser@ieee.org).}}

\maketitle

\begin{abstract}
    In this paper, we introduce and study different dissipativity notions and different turnpike properties for discrete-time stochastic nonlinear optimal control problems. The proposed stochastic dissipativity notions extend the classic notion of Jan C. Willems to $\mathit{L^r}$ random variables and to probability measures.
    Our stochastic turnpike properties range from a formulation for random variables via turnpike phenomena in probability and in probability measures to the turnpike property for the moments.  Moreover, we investigate how different metrics (such as Wasserstein or Lévy-Prokhorov) can be leveraged in the analysis.  Our results are built upon stationarity concepts in distribution and in random variables and on the formulation of the stochastic optimal control problem as a finite-horizon Markov decision process.
    We investigate how the proposed dissipativity notions connect to the various stochastic turnpike properties and we work out the link between 
    different forms of dissipativity.
\end{abstract}


\section{Introduction}
    \IEEEPARstart{D}{issipativity} and strict dissipativity as introduced by Willems \cite{willems1972a,willems1972b} and the so-called turnpike property---first observed by Ramsey \cite{Ramsey1928} and von Neumann \cite{Neumann1945}---are important tools for analyzing the qualitative properties of optimally controlled systems. While dissipativity states that a system cannot contain more energy than supplied from the outside, the turnpike property describes the phenomenon that optimal and near-optimal trajectories spend most of the time near a particular solution, which in the simplest case is an optimal steady state of the system. The connection of these properties to model predictive and receding horizon control \cite{Faulwasser2018, Faulwasser2022, Gruene2022, Gruene2013, Breiten2020, Angeli2011, Mueller2014, Koehler2018} has fueled recent research activities. The topic also received attention in fields such as optimal control with partial differential equations \cite{Gugat2021,Gugat2023,Trelat2018,Zuazua2017}, shape optimization \cite{Lance2020}, or mean field games \cite{Porretta2018}.
    
    For deterministic optimal control problems, the relation between dissipativity and the turnpike property \JS{is} well understood. Loosely speaking, under mild condition strict dissipativity implies the turnpike property while the converse is true under somewhat more restrictive conditions \cite{Faulwasser2017, Gruene2018, Gruene2016}. 
    For stochastic optimal control problems, this relation is so far only very little explored.
    One of the challenges that arises here is to identify the appropriate probabilistic objects for which such an analysis can be carried out.
    Two possible such objects are the distributions and the moments of the solutions of the optimally controlled system, and most of the current work about stochastic turnpike properties are using one of these, e.g.\ \cite{Kolokoltsov2012, Marimon1989, Sun2022}. The approach via distributions is also closely related to a probability \JS{measure-based} definition of dissipativity \JS{introduced in \cite{Gros2022} for general Markov decision processes and in \cite{Kordabad2022} for discounted stochastic optimal control problems}.
    While this approach can be used to characterize the qualitative behavior of the probability distribution and moments of optimal solutions, it does not allow to obtain {\em pathwise} information, since 
    two random variables with the same distribution are indistinguishable through the lens of moments and measures. 
    \JS{However, recent numerical results in \cite{Ou2021} show that pathwise turnpike properties can expected to be observed in \JS{measurements from} real-world applications.}
    
    Motivated by this observation, the authors recently introduced an alternative dissipativity notion for discrete-time linear-quadratic (LQ) optimal control problems, using the representation of the stochastic states and controls via random variables, see \cite{CDCPaper, CDCExtension}. 
    It turns out that this notion can be used to analyze both the distributional {\em and} the pathwise behavior of the problem. Furthermore, we have shown that the stationary distribution in our approach is characterized by the same optimization problem as for the probability \JS{measure-based} approach, cf.\ \cite[Remark 5.3]{CDCExtension}.
    Nevertheless, the exact connection between these two dissipativity notions is not yet clear and the available analysis is so far limited to LQ optimal control problems.

    \JS{The need to analyze stochastic turnpike and dissipativity properties is of particular application interest due to the deep connection to model predictive control (MPC). Typically, dissipativity and turnpike assumptions are at the core of the analysis of MPC schemes in terms of their performance and stability.
    In \cite{PerformanceCDC} it was already shown that stochastic turnpike properties can be used to obtain near-optimal performance estimates for stochastic economic MPC, and we strongly conjecture that the presented notions of stochastic dissipativity and turnpike can also be used to prove corresponding stochastic stability properties. First stability results for stochastic MPC have already been presented, for example, in \cite{Sopasakis2017,Mcallister2022} and in \cite{Kordabad2022} using stochastic dissipativity.} 
    
    Thus, the contribution of this paper is twofold: \JS{W}e investigate 
    the relationship between these two dissipativity concepts and 
    we carry out a dissipativity and turnpike analysis for general discrete-time nonlinear stochastic optimal control problems. To this end, we introduce four types of turnpike properties---in $L^r$, pathwise in probability, in distribution, and for the moments---and show that they form a hierarchy.
    Crucially, the metric used for the distributions (e.g., Wasserstein or Lévy-Prokhorov) determine the details of the considered hierarchy. 
    We furthermore investigate how the strict dissipativity concepts---for the random variable and for distributions---relate to each other and which of the four turnpike properties they imply. Figure \ref{fig:sketch_of_results}, \JS{at the end of the paper}, gives an overview of the implications we obtain. Compared to the LQ case, the proofs for the nonlinear case require a completely different way of reasoning, since the tools from LQ optimal control like the superposition principle and the Riccati equation are not available.
    
    The remainder of the paper is structured as follows. Section~\ref{sec:Setting} introduces the stochastic optimal control problems under consideration and the different forms of dissipativity and turnpike behavior that we study. Moreover, we define the stationarity concepts used later. 
    In Section~\ref{sec:MainResults}, we present the main results. 
    We show which turnpike properties are induced by the different dissipativity notions and analyze the connection of these notions. 
    Furthermore, we show that the stationary distribution needed in the strict dissipativity and turnpike formulations can be characterized by an optimality condition.
    The theoretical findings are illustrated by a nonlinear example in Section~\ref{sec:Example} while Section~\ref{sec:Conclusion} concludes the paper.

\section{Setting and definitions} \label{sec:Setting}

    Before presenting the main results of this paper in Section~\ref{sec:MainResults}, we start by introducing the stochastic optimal control problems under consideration. Moreover, we define the different concepts of stochastic dissipativity and turnpike used for our investigations.

    \subsection{The stochastic optimal control problem}
    We consider discrete-time stochastic control systems described by a function
    \begin{equation*}
        f: \X \times \U \times \W \rightarrow \X, \quad (x,u,w) \mapsto f(x,u,w),
    \end{equation*}
    which is continuous in $(x,u) \in \X \times \U$ for almost all realizations $w \in \W$, where $\X$, $\U$, and $\W$ are separable Banach-spaces. The control system is then defined by the iteration
    \begin{equation} \label{eq:stochSys}
        X(k+1) = f(X(k),U(k),W(k)), \quad X(0) = X_0,
    \end{equation}
    where $W(k)$ is independent of $X(k)$ and $U(k)$ and $(W(0),W(1),\ldots)$ is an \emph{i.i.d.} sequence of random variables with distribution $\varrho_W$.
    \JS{Let $\normX{\cdot}$ be an arbitrary norm on $\mathcal{X}$ such that the space $\mathcal{X}$ equipped with this norm is a Banach-space.}
    In order to define the stochastic state and input spaces for the system, 
    we \JS{consider a probability space $(\Omega,\F,\Prob)$} and denote by $\Llp[n]{\X}$ the space of $\F$-measurable random variables $X: \Omega \rightarrow \X$ with finite $n$-th moments, 
    \JS{such that the norm}
    \begin{equation*}
        \Vert X \Vert_{L^n} := \Exp{\normX{X}^n}^{1/n} = \left( \int_{\Omega} \normX{X}^n ~ \mbox{d}\Prob \right)^{1/n}
    \end{equation*}
    \JS{exists}.
    We assume $X(k) \in \Lp[n]{\X}$ and $U(k) \in \Lp[l]{\U}$, where $(\F_k)_{k \in \N_0}$ is the smallest filtration such that $X$ is an adapted process, i.e.
    \begin{equation*}
        \F_k = \sigma(X(0),\ldots,X(k)) \subseteq \F, \quad k \in \N_0.
    \end{equation*}
    This choice of the stochastic filtration induces a causality requirement, which ensures that the control action $U(k)$ at time $k$ only depends on the initial condition and on the sequence of past disturbances $(W(0),\ldots,W(k-1))$ but not on future events. 
    We refer to \JS{\cite{Fristedt1997, Protter2005}} for more details on stochastic filtrations.
    In addition to the filtration condition we impose constraints $(X,U) \in \Y \subseteq \Llp[n]{\X} \times \Llp[l]{\U}$ on the stochastic state $X$ and control $U$ which are assumed to be law-invariant, i.e., if $(X,U) \in \Y$ then $(Y,V) \in \Y$ also holds for all $(X,U) \sim (Y,V)$. Here $(X,U) \sim (Y,V)$ means that the pair $(X,U)$ has the same joint distribution as $(Y,V)$. 
    Law-invariant constraints can be, for instance, chance constraints, constraints on single moments, or constraints associated to law-invariant risk measures like the value at risk.
    We define the sets 
    \begin{equation*}
    \begin{split}
        \XX := \lbrace X \in &\Llp[n]{\X} \mid \\
        &\exists U \in \Llp[\JS{l}]{\U}  ~:~ (X,U) \in \Y \rbrace \phantom{.}
    \end{split}
    \end{equation*}
    and
    \begin{equation*}
    \begin{split}
        \UU := \lbrace U \in &\Llp[\JS{l}]{\U} \mid \\ 
        &\exists X \in \Llp[n]{\X} ~:~ (X,U) \in \Y \rbrace.
    \end{split}
    \end{equation*}
    A control sequence $\mathbf{U} := (U(0),\ldots,U(N-1)) \in \UU^N$ is called admissible for $X_0 \in \XX$ if $U(k) \in \Lp[l]{\U}$, $(X(k),U(k)) \in \Y$ for all $k=0,\ldots,N-1$, and $X(N) \in \XX$. 
    In this case, the corresponding trajectory is also called admissible. 
    The set of admissible control sequences is denoted by $\UU^N(X_0)$. 
    Likewise, we define $\UU^{\infty}(X_0)$ as the set of all control sequences $\mathbf{U} \in \UU^{\infty}$ with $U(k) \in \Lp[l]{\U}$ and $(X(k),U(k)) \in \Y$ for all $k \in \N_0$. 
    For technical simplicity, we assume that $\XX$ is control invariant, i.e., that $\UU^{\infty}(X_0) \neq \emptyset$ for all $X_0 \in \XX$.
    For an initial value $X_0 \in \XX$ \JS{and} a control sequence $\mathbf{U} \in \UU^{N}(X_0)$, 
    we denote the trajectories of \eqref{eq:stochSys} as $\Xtraj{\cdot}$ or short $\XtrajShort{\cdot}$ if there is no ambiguity about $X_0$ and $\mathbf{U}$.
    \JS{Note, that the solution $\Xtraj{\cdot}$ also depends on the disturbance $\mathbf{W}$. 
    However, for the sake of readability, we do not highlight this in our notation and assume in the following that $\mathbf{W} := (W(0),\ldots, W(N-1))$ is an arbitrary but fixed stochastic process.}
    
    In order to extend the stochastic system \eqref{eq:stochSys} to a stochastic optimal control problem, we consider stage costs 
    $\ell: \Y \rightarrow \R$, 
    which are also assumed to be law-invariant, i.e., for any two pairs $(X,U)$ and $(Y,V)$ it holds that $\ell(X,U) = \ell(Y,V)$ if $(X,U) \sim (Y,V)$.
    \JS{The class of law-invariant cost includes the expected cost $\E[g(X,U)]$ for some function $g:\mathcal{X} \times \mathcal{U} \rightarrow \mathcal{X}$, which is widely used in stochastic optimal control, but can also involve additional variance penalization or other kinds of risk aware formulations.}
    
    Then, in summary, the stochastic optimal control problem on horizon $N \in \N \cup \{\infty\}$ under consideration reads
    \begin{equation} \label{eq:stochOCP1}
        \begin{split}
            \min_{\mathbf{U} \in \UU^{N}(X_0)} &J_N(X_0,\mathbf{U}) := \sum_{k=0}^{N-1} \ell(X(k),U(k)) \\
            s.t. ~ X(k+1) &= f(X(k),U(k),W(k)), ~ X(0) = X_0 
        \end{split}
    \end{equation}

    \subsection{Markov policies and stationarity}
    In stochastic settings, the concept of an equilibrium of a deterministic system has to be extended to a more general stationarity concept.
    Hence, a structured analysis of stationary solutions to \eqref{eq:stochSys} is pivotal for our further developments.
    It is known, see, e.g.\,  \cite{CDCPaper,CDCExtension}, and in fact quite obvious that in general a pair of constant random variables $(X,U) \in \Llp[n]{\X} \times \Llp[l]{\U}$ is not a steady state of \eqref{eq:stochSys} if the system is subject to persistent excitation by the noise. 
    A natural remedy here is to look for constant distributions rather than for constant solutions and this is the first form of stationarity we consider in this paper. 
    However, as we want to analyze the behavior of the random variable $X(k)$ as an element of $\Llp[n]{\X}$ and also pathwise, looking at distributions alone does not provide sufficient information. Hence, in the work at hand we also replace the deterministic steady states by stationary processes $X(k)$ and $U(k)$, i.e., by solutions of \eqref{eq:stochSys} that vary with time $k$ but whose distributions are constant over time.

    In order to formalize this, we first need to show that for obtaining optimal stationary solutions---or, in fact, general optimal solutions---we can restrict ourselves to control sequences that do not depend on $X(0),\ldots,X(k)$ but only on $X(k)$.
    \JS{To this end, we first notice that for every control process $U\args$ which is adapted to the filtration $(\F_k)_{k \in \N_0}$ , i.e. $\sigma(U(k)) \subseteq \F_k$, we can find a measurable function $\pi_k: \X^{k+1} \rightarrow \U, k \in \N_0$, also called feedback strategy or policy, such that $U(k) = \pi_k(X(0),\ldots,X(k))$
    holds for all  $k \in \N_0$. This well-known result, see for instance \cite[Lemma~1.14]{Kallenberg2021}, shows that we can equivalently minimize over sequences $\boldsymbol{\pi} = (\pi_0,\pi_1,\ldots)$ instead of over adapted control processes $U\args$ satisfying the assumed filtration condition.
    However, since we demand that $U(k) \in \Lp[l]{\U} \subseteq \Llp[l]{\U}$, we have to ensure that $U(k) = \pi_k(X(0),\ldots,X(k)) \in \Llp[l]{\U}$
    for all $k \in \N_0$ when switching to feedback strategies.
    Thus, we introduce the set $\Pi_k(\X,\U)$ given by 
    \begin{equation*}
    \begin{split}
        \Pi_k(\X,\U) := \lbrace \pi : \X^{k+1} \rightarrow \U \mid \pi \circ X \in \Llp[l]{\U} \qquad \\
        \mbox{ for all } X \in (\Llp[n]{\X})^{k+1} \rbrace.
    \end{split}
    \end{equation*}
    Although we can now use feedback strategies instead of control processes, the problem that the admissible set depends on the time still remains, since $\Pi_k \subset \Pi_{k+1}$ holds for all $k \in \N_0$.
    However, using \cite[Theorem 6.2]{Altman2021} one can prove that the solution of the optimal control problem \eqref{eq:stochOCP1} can always be characterized by a sequence of so-called Markov policies, i.e., policies $\boldsymbol{\pi}=(\pi_0,\ldots,\pi_{N-1})$ with $\pi_k \in \Pi_0(\X,\U)$ for all $k=0,\ldots,N-1$. Here the name ``Markov policy'' for $\boldsymbol{\pi}$ refers to the fact that for $U(k) = \pi_k \circ X(k)$ the system \eqref{eq:stochSys} defines a Markov chain.}

    This restriction has the advantage that we can replace the time dependent filtration condition by constraining the controls to Markov policies, which only depend on the current state and not the whole history of the states. 
    Markov policies greatly simplify the analysis of the optimally controlled system because we know that for a Markov \JS{policy} $\pi$, the system $$X(k+1) = f(X(k),\pi(X(k)),W(k))$$
    defines a Markov chain. This directly implies that for every Markov policy $\pi$ there is a transition probability $p_{\pi}: \X \rightarrow \mathcal{P}(\JS{\X})$ only depending on the distribution of the noise such that $\Prob (x(k+1) \in A \vert x(k) = y) = p_{\pi}(y,A)$
    for every $A \in \mathcal{B}(\X)$, almost every $x \in \X$ and every $k \in \N$, see \cite{Meyn1993}. 
    \JS{Here $\mathcal{P}(\X)$ denotes the set of all probability measures on $\X$ and $\mathcal{B}(\X)$ is the Borel $\sigma$-algebra on $\X$.}
    Thus, for  $X_k \sim \varrho_k$ and any Markov policy $\pi \in \Pi_0(\X,\U)$ we can compute the distribution $\varrho_{k+1}$ of $X(k+1)$ using the transition operator given by 
    \begin{equation}\label{eq:kernel}
        \varrho_{k+1} = \mathcal{T}(\pi,\varrho_k) = \int_{\X} p_{\pi}(y,\pi(y)) \varrho_k(dy).
    \end{equation}
    Put differently, for a given policy $\pi$ the distribution $\varrho_{k+1}$ only depends on the distribution $\varrho_{k}$ but not on the particular form of $X(k)$ as a random variable. 
    In turn, this implies that for a given Markov policy the forward propagation of the distribution of the state does also only depend on the current distribution. 
    Hence, $f(X,\pi(X),W) \sim f(\widetilde X,\pi(\widetilde X),W)$
    holds for all $\widetilde X \sim X$ and $W \sim \rho_W$ independent of the particular choice of $W$,  $X$ and $\widetilde{X}$. 
    
    We want to use \JS{these} properties in the remainder of the paper. Thus, henceforth, the optimal control problem under consideration reads
    \begin{equation} \label{eq:stochOCP}
        \begin{split}
            \min_{\mathbf{U} \in \FF^{N}(X_0)} &J_N(X_0,\mathbf{U}) := \sum_{k=0}^{N-1} \ell(X(k),U(k)) \\
            s.t. ~ X(k+1) &= f(X(k),U(k),W(k)), ~ X(0) = X_0 
        \end{split}
    \end{equation}
    which is problem \eqref{eq:stochOCP1} restricted to the set of Markov policies \JS{$\FF^N(X_0) \subseteq \UU^N(X_0)$ defined as} 
    \begin{equation*}
    \begin{split}
        \FF^N(X_0) := \{\mathbf{U} \in &\UU^N(X_0) \mid \exists \pi_k \in \Pi_0(\X,\U): \\
        &\qquad U(k) = \pi_k(X(k)), ~ k=0,\ldots,N \}.
    \end{split}
    \end{equation*}
    \begin{rem} \label{rem:opt_pi}
\JS{Each $\pi_k$ in the definition of $\FF^N(X_0)$ maps the values of $X(k)$ pointwise to the values of $\pi_k(X(k))$, i.e., if we consider $\pi_k(X(k))$ as a random variable, then $\pi_k(X(k))(\omega) = \pi_k(X(k,\omega))$. Nevertheless, an optimal $\pi_k$ as part of a sequence $\mathbf{U}$ minimizing \eqref{eq:stochOCP}, in general, also depends on the distribution of $X(k)$. In other words, for two optimal solutions $X(k)$ and $\widetilde X(k)$ with different distributions at time $k$, we will in general obtain two different optimal feedback laws $\pi_k\ne \tilde \pi_k$ at time $k$.}

\JS{
Only in exceptional cases $\pi_k$ will be independent of $X(k)$. This happens, e.g., when there are no constraints and when the cost is of the form $\ell(X,U)= \E[g(X,U)]$, as in linear-quadratic optimal control. This is because in this case minimization of the expectation is equivalent to minimization of the individual solutions starting in $X(k,\omega)$, $\omega\in\Omega$. However, this equivalence does not hold in general.}
    \end{rem}
    We emphasize again that the set of optimal solutions to \eqref{eq:stochOCP1} and \eqref{eq:stochOCP} coincides, although  \eqref{eq:stochOCP} in general has a smaller set of feasible solutions\JS{.} 
    
    Now we have all the technical prerequisites to introduce the two different concepts of stationarity for our stochastic system that we deal with in this paper.  
    To this end, we define the set 
    \begin{equation*}
        \PP_n(\X) := \left\lbrace \mu \in \PP(\X) ~ \middle| ~ \int_{\X} \Vert x \Vert^{n} \mu(\mbox{d}x) < \infty \right\rbrace,
    \end{equation*}
    where $\PP(\X)$ denotes the set of all probability measures. 
    It is directly clear that for every random variable $X \in \Llp[n]{\X}$ with $X \sim \varrho$ we have $\varrho \in \PP_n(\X)$ and, thus, the set $\PP_n(\X)$ has a natural connection to $\Llp[n]{\X}$.
    The next definition now formalizes stationarity in \JS{the} sense of distributions. \JS{This stationarity concept is well established in stochastic control and the probability measure defining the stationary distribution is also known as invariant measure, see for instance \cite{Altman2021,Meyn1989,Hernandez2012}.}

    \begin{defn}[Stationary policy and distribution] \label{defn:stationaryDistr}
        A distribution-policy pair $(\varrho^s_X,\pi^s)$ with $\varrho^s_X \in \PP_n(\X)$ and $\pi^s \in \Pi_0(\X,\U)$ is called stationary for system \eqref{eq:stochSys} if for all $X \sim \varrho^s$ and disturbances $W \sim \varrho_W$ it holds that $f(X,\pi^s(X),W) \sim \varrho^s_X$.
    \end{defn}

    Note that Definition~\ref{defn:stationaryDistr} looks at a one step problem and, hence, only considers Markov policies $\pi^s \in \Pi_0(\X,\U)$.
    Therefore, as already mentioned above, from \cite{Meyn1993} we know that there is a transition operator $\mathcal{T}: \PP_n(\X) \times \Pi_0(\X,\U) \rightarrow \PP_n(\X)$ only depending on the distribution $\varrho_W$ of the noise which can be used to directly forward propagate the distribution of the states. 
    This means that we can alternatively define the pair $(\varrho^s_X,\pi^s)$ as a steady state of the transition operator, i.e.
    \begin{equation} \label{eq:StatTransition}
        \mathcal{T}(\varrho_X^s,\pi^s) = \varrho_X^s.
    \end{equation}
    This also shows that Markov policies are sufficient to characterize stationary distributions since we can iteratively reapply $\pi^s$ to keep the distribution constant for all $k \in \N_0$. \JS{The crux of using transition operators~\eqref{eq:kernel} is the difficulty to state them in explicit form.}
    
    Since we also want to analyze the behavior of solutions in $\JS{\Llp[n]{\X}}$ as well as pathwise behaviors, we also make use of a stationarity concept of the random variables forming the solutions on the probability space $(\Omega,\F,\Prob)$. 
    To this end, we define the subset $\Gamma(\X)$ of $\PP_n(\X)$ by
    \begin{equation*}
        \Gamma(\X) := \lbrace \varrho \in \PP_n(\X) \mid \exists X \in \Llp[n]{\X}:~X \sim \varrho \rbrace 
    \end{equation*}
    and assume $(\varrho^s_X,\pi^s) \in \Gamma(\X) \times  \Pi_0(\X,\U)$ to ensure that we can always find a random variable $X \sim \varrho^s_X$.
    The \JS{next} definition now extends the concept of stationarity from Definition~\ref{defn:stationaryDistr} to a stochastic process. The definition is equivalent \JS{to the one given in \cite{Doob1953}, since the iteration \eqref{eq:sys_stat} defines a Markov-chain. However, we state it in a form that is more convenient for our pathwise considerations than the definition by the joint distributions of the states for different time instants in \cite{Doob1953}.}

    \begin{defn}[Stationary stochastic processes] \label{defn:stationaryProcess}
        Given a sequence of disturbances $\mathbf{W} = (W(0),W(1),\ldots)$, the pair of the stochastic processes $(\mathbf{X}^s,\mathbf{U}^s)$ given by
        \begin{equation} \label{eq:sys_stat}
            \Xstraj{k+1} = f(\Xstraj{k}, U^s(k), W(k))
        \end{equation}
        with $\mathbf{U}^s \in \FF^{\infty}(X^s(0))$ is called stationary for system \eqref{eq:stochSys} if 
        \begin{equation*}
        \begin{split}
            \JS{\Xstraj{k}} \sim \varrho^s_X, \quad U(k) \sim \varrho^s_U, \quad (\Xstraj{k},U^s(k)) \sim \varrho^s_{X,U}
        \end{split}
        \end{equation*}
        for all $k \in \N_0$. 
    \end{defn}
    For any stationary stochastic process $(\mathbf{X}^s,\mathbf{U}^s)$ in the sense of Definition \ref{defn:stationaryProcess} the corresponding distribution $\varrho^s_X$ is stationary in the sense of Definition \ref{defn:stationaryDistr}, i.e.\ a steady state of the transition operator \eqref{eq:StatTransition}. At the first glance this is not fully obvious, because the $\pi_k$ defining $\mathbf{U}^s\in\FF^{\infty}(X^s(0))$ in Definition \ref{defn:stationaryProcess} may depend on time while $\pi^s$ in Definition \ref{defn:stationaryDistr} does not. However, because of equation \eqref{eq:kernel} each $\pi_k$ has the property required for $\pi^s$ in Definition \ref{defn:stationaryDistr}.
    
    Conversely, for any stationary distribution there exists a stationary stochastic process\JS{.} 
    \JS{In particular, for a stationary feedback $\pi^s \in \Pi_0(\X,\U)$ with stationary distribution $\varrho_X^s \in \Gamma(\X)$ such that $(X^s,\pi^s(X^s)) \in \Y$ for $X^s \sim \varrho^s_X$, the pair $(\mathbf{X}^s,\mathbf{U}^s)$ given by 
    \begin{equation} \label{eq:equiDistrStat}
    \begin{split}
        \Xstraj{k+1} &= f(\Xstraj{k}, U^s(k), W(k)), \\
        U^s(k) &= \pi^s(\Xstraj{k})
    \end{split}
    \end{equation}
    is stationary for system \eqref{eq:stochSys} for any $\Xstraj{0} \sim \varrho_X^s$.}
    %
    This shows that Definition~\ref{defn:stationaryProcess} is a natural extension of the stationarity concepts defined by the transition operator \eqref{eq:StatTransition} to stochastic processes. The main difference between the stationarity in distribution usually used in the context of Markov decision processes and Definition~\ref{defn:stationaryProcess} is that we impose the additional condition that the state process representing the stationary distribution follows the evolution of the stochastic system \eqref{eq:stochSys}.

    \subsection{Dissipativity and strict dissipativity}
    The goal of this paper is to define (strict) dissipativity notions for nonlinear stochastic optimal control problems that allow to conclude a wide range of stochastic turnpike properties. The next two definitions provide these notions. 
    The first one \JS{which is based on \cite{Gros2022}} deals with the distributional behavior of the problem, while the second one is formulated for the random variables defining the solution process, such that we can gain information about the realization paths of the solutions.
    \JS{To this end, we use the class of comparison functions
    \begin{equation*}
        \begin{split}
            \K_{\infty} := \{ \alpha : \R^+_0 \rightarrow \R_0^+ ~\mid~ & \alpha \mbox{ continuous, strictly increasing,} \\
            &\mbox{and unbounded with } \alpha(0)=0 \}.
        \end{split}
    \end{equation*}
    }%
    \JS{Furthermore, w}e denote by $P_X : \mathcal{B}(\X) \rightarrow [0,1] \in \PP_n(\X)$ the push forward measure of the random variable $X \in \Llp[n]{\X}$ which characterizes its distribution and is defined by 
    $P_X(A) := \JS{\Prob}(X^{-1}(A))$ for all $A \in \mathcal{B}(\X)$.
    Note that for a stationary process with distribution $\varrho_X^s$ we have $P_{\Xstraj{k}} = \varrho_X^s$ for all $k \in \N_0$.
    For a Markov \JS{policy} $\pi \in \Pi_0(\X,\U)$ we abbreviate the push forward measure of the subsequent state $X^+ := f(X,\pi(X),W)$ as $P_{X^+}$, since this measure does not depend on the exact representation of $W$.

    \begin{defn}[Distributional dissipativity] \label{defn:DissiDistr}
        Given a stationary feedback $\pi^s$ and a stationary distribution $\varrho_X^s$, we call the stochastic optimal control problem \eqref{eq:stochOCP} strictly distributionally dissipative at $\varrho_X^s$, if there exists a storage function $\Lambda: \mathcal{P}(\X) \rightarrow \R$ bounded from below, a metric $d_D$ on the space $\PP_n(\X)$, and a function $\alpha \in \K_{\infty}$ such that 
        \begin{equation} \label{eq:DissiIneqDistr}
        \begin{split}
            \ell(X,\pi(X)) - \ell(X^s,\pi^s(X^s)) + \Lambda(&P_X) - \Lambda(P_{X^+}) \\
            &\geq \alpha(d_D(P_X, \varrho^s_X))
        \end{split}
        \end{equation}
        holds for all $(X,\pi(X)) \in \Y$ and $X^s \sim \varrho^s_X$. The system is called distributionally dissipative, if inequality \eqref{eq:DissiIneqDistr} holds with $\alpha \equiv 0$.
    \end{defn}
    
    Since our stage costs and constraints are law-invariant, we can define  $\hat{\ell}(\varrho,\pi) := \ell(X,\pi(X))$ for $X \sim \varrho$ and also consider the optimal control problem in probability measures, leading to the finite-horizon Markov decision problem
    \begin{equation} \label{eq:stochOCPmeasure}
        \begin{split}
            \min_{\boldsymbol{\pi} \in \St^N(\varrho_0)} \hat{J}_N(\rho_0,&\boldsymbol{\pi}) :=\sum_{k=0}^{N-1} \hat{\ell}(\varrho_k,\pi_k) \\
            s.t. ~ \varrho_{k+1} &= \mathcal{T}(\varrho_k,\pi_k), ~\varrho_0 = \rho_0
        \end{split}
    \end{equation}
    with 
    \begin{equation*}
    \begin{split}
        &\St^N(\varrho_0) := \{ \boldsymbol{\pi} \in \Pi_0(\X,\U)^N \mid \\
        &\quad \mathbf{U} = (\pi_0(X(0)),\ldots,\pi_{N-1}(X(N-1))) \in \FF^N(X_0 \sim \rho_0) \} 
    \end{split}
    \end{equation*}
    for some initial distribution $\rho_0 \in P_{\XX} := \{ \varrho \in \Gamma(\X) \mid X \in \XX \mbox{ for } X \sim \varrho \}$. This problem yields the same optimal strategies as problem \eqref{eq:stochOCP} for $\varrho_0 = P_{X_0}$.
    Since in Definition~\ref{defn:DissiDistr} we only consider Markov policies, the  strict dissipation inequality \eqref{eq:DissiIneqDistr} can thus be written as 
    \begin{equation*}
        \hat{\ell}(\varrho_k,\pi) - \hat{\ell}(\varrho_X^s,\pi^s) + \Lambda(\varrho_k) - \Lambda(\varrho_{k+1}) \geq \alpha(d_D(\varrho_{k},\varrho^s )),
    \end{equation*}
    which is the dissipativity notion based on probability measures introduced in \cite{Gros2022}, except that 
    \JS{we restrict ourselves to metrics on probability measures instead of dissimilarity measures. This is due to the fact that we want to be able to distinguish two probability measures by their distance $d_D$, which is not possible for a dissimilarity measure $\hat{d}_D$ since it also allows $\hat{d}_D(\varrho_1,\varrho_2)=0$ for $\varrho_1 \neq \varrho_2$.}
    \JS{Note that this concept of dissipativity is related to the approach of rewriting the stochastic optimal control problem as a deterministic one in spaces of probability measures, cf.\ \cite[Section~9.2]{Bertsekas1978}, and then using classical dissipativity notions for deterministic systems in infinite dimension.}
    For the purpose of this paper, Definition~\ref{defn:DissiDistr} simplifies the transition between the distributional and the random variable perspective.
    This is important since next we also propose a dissipativity notion formulated in random variables, which takes the paths of the stochastic processes into account formalized in the following definition. 
    \JS{This definition extends the $L^2$ dissipativity introduced in \cite{CDCPaper,CDCExtension} to general $L^r$ spaces.}
    Here and in the following, we use that for a stationary process $(\mathbf{X}^s,\mathbf{U}^s)$ due to the law-invariance of $\ell$ the value $\ell(\Xstraj{k},U^s(k))$ is independent of $k$. We denote this value by $\ell(\mathbf{X}^s,\mathbf{U}^s)$.

    \begin{defn}[$L^r$ dissipativity in random variables] \label{defn:DissiLp}
        Consider a pair of stationary stochastic processes $(\mathbf{X}^s,\mathbf{U}^s)$ for a given noise sequence $\mathbf{W} = (W(0),W(1),\ldots)$ according to Definition~\ref{defn:stationaryProcess}.
        Then we call the stochastic optimal control problem \eqref{eq:stochOCP} strictly dissipative in $L^r$ at $X^s$, if there exists a storage function $\lambda: \N_0 \times \Llp[n]{\X} \rightarrow \R$ bounded from below and a function $\alpha \in \K_{\infty}$ such that 
        \begin{equation} \label{eq:DissiIneqLp}
        \begin{split}
            &\ell(\XtrajShort{k},U(k)) - \ell(\mathbf{X}^s,\mathbf{U}^s) \\
            &\qquad \quad + \lambda(k,\XtrajShort{k}) - \lambda(k+1,\XtrajShort{k+1}) \\ 
            &\geq \alpha(\Exp{\normX{ X(k,\mathbf{W}) - X^s(k,\mathbf{W}) }^r}) \\
            &\qquad \quad =  \alpha(\Vert \XtrajShort{k} - \XstrajShort{k} \Vert_{L^r}^r)
        \end{split}
        \end{equation}
        holds for all $k \in \N_0$, $X_0 \in \XX$ and $\mathbf{U} \in \FF^{\infty}(X_0)$. The system is called $L^r$ dissipative if inequality \eqref{eq:DissiIneqLp} holds with $\alpha \equiv 0$.
    \end{defn}
    \begin{rem}[The role of $L^r$ in \eqref{eq:DissiIneqLp}]
    \JS{In case of strict $L^r$ dissipativity, the monotonicity  of $\alpha \in \mathcal K$ combined with $\Llp[r]{\X} \subset \Llp[k]{\X}, k>r$ gives that if \eqref{eq:DissiIneqLp} holds for $L^{k}$ is also holds for $L^{r}$ with $r \leq k$.}

    \JS{In case of non-strict dissipativity ($\alpha \equiv 0$), the dissipation inequality \eqref{eq:DissiIneqLp} does not depend on any norm of $\Llp[r]{\X}$. Hence non-strict $L^{r}$ dissipativity for some $r \in \{1, \dots, n\}$ implies non-strict $L^{r}$ dissipativity for all $r \in \{1, \dots, n\}$.}
    \end{rem}
    \begin{rem}[Non-Markovian extensions]
        Definitions~\ref{defn:stationaryProcess}~and~\ref{defn:DissiLp} can be easily extended to the non-Markovian setting by considering $\UU^N(X_0)$ instead of $\FF^N(X_0)$. 
        However, in this case the storage function $\lambda$ from Definition~\ref{defn:DissiLp} may implicitly depend on the whole history of the state process due to the underlying filtration. 
        The reason for not considering this generalization is that such history-dependent $\lambda$ significantly complicate the analysis while not providing any benefit in our setting, because the optimal policies are Markov policies\JS{.} 
        Moreover, for Definition~\ref{defn:DissiDistr} such a generalization is not possible, since for its formulation it is crucial that the information of the current distribution and feedback law is sufficient to compute the distribution of the subsequent state, which is only possible in the Markov setting using the transition operator \eqref{eq:kernel}.
    \end{rem}

    \subsection{Turnpike properties}

    At last we introduce four types of turnpike properties \JS{which are generalizations of the one presented in \cite{CDCPaper,CDCExtension} for the stochastic LQ case}. Examining their connection with the dissipativity notions from Definition~\ref{defn:DissiDistr} and Definition~\ref{defn:DissiLp} is the main purpose of this paper. We will do this in the next section. 
    Note that in the following definition the metric $d_D$ used for the probability measures in part \ref{defn:TurnpikeDistr}) is arbitrary. However, some of our results will only hold for specific metrics, which we define in Definition~\ref{def:wasserstein_weak}.

    \begin{defn}[Stochastic turnpike properties] \label{defn:stochTurnpike}
        Consider a pair of stationary stochastic processes $(\mathbf{X}^s,\mathbf{U}^s)$ for a given noise sequence $\mathbf{W} = (W(0),W(1),\ldots)$ according to Definition~\ref{defn:stationaryProcess} and set $\mathcal{N} = \{0,\ldots,N\}$. Then the stochastic optimal control problem is said to have the
        \begin{enumerate}
            \item \textbf{$L^r$ turnpike property} for some $r \leq n$ if \JS{there exists $\alpha_1 \in \K_{\infty}$ such that the following holds: For each $X_0 \in \XX$ there exists a $C_a > 0$} 
            such that for each $\delta > 0$ and $N \in \N$, each control sequence $\mathbf{U} \in \FF^N(X_0)$ satisfying the condition $J_N(X_0,\mathbf{U}) \leq \delta + N \ell(\mathbf{X}^s,\mathbf{U}^s)$ and each $\eps > 0$ the value 
            \begin{equation*}
            \begin{split}
                \Lset^r_{\eps} := \# \lbrace &k \in \mathcal{N} \mid \Exp{\normX{ \Xtraj{k} - \XstrajShort{k} }^r} \leq \eps \rbrace
            \end{split}
            \end{equation*}
            satisfies the inequality $\Lset^r_{\eps} \geq N - (\delta + C_a)/\alpha_1(\eps)$.
            \label{defn:TurnpikeLp}
            \item \textbf{pathwise-in-probability turnpike property} if \JS{there exists $\alpha_2 \in \K_{\infty}$ such that the following holds: For each $X_0 \in \XX$ there exists a $C_b > 0$} 
            such that for each $\delta > 0$ and $N \in \N$, each control sequence $\mathbf{U} \in \FF^N(X_0)$ satisfying the condition $J_N(X_0,\mathbf{U}) \leq \delta + N \ell(\mathbf{X}^s,\mathbf{U}^s)$ and each $\eps > 0$ the value 
            \begin{equation*}
            \begin{split}
                \Pset_{\eps} := \# \lbrace k &\in \mathcal{N} \mid \\
                &\Prob (\normX{ \Xtraj{k} - \XstrajShort{k} } \leq \eps) \geq 1-\eps \rbrace
            \end{split}
            \end{equation*}
            satisfies the inequality $\Pset_{\eps} \geq N - (\delta + C_b)/\alpha_2(\eps)$. 
            \label{defn:TurnpikeProb}
            \item \textbf{distributional turnpike property} if \JS{there exists $\alpha_3 \in \K_{\infty}$ and a metric $d_D$ on the space $\PP_n(\X)$ such that the following holds: For each $X_0 \in \XX$ there exists a $C_c > 0$}
            such that for each $\delta > 0$ and $N \in \N$, each control sequence $\mathbf{U} \in \FF^N(X_0)$ satisfying the condition $J_N(X_0,\mathbf{U}) \leq \delta + N \ell(\mathbf{X}^s,\mathbf{U}^s)$ and each $\eps > 0$ the value 
            \begin{equation*}
                \Dset_{\eps} := \# \left\lbrace k \in \mathcal{N} \mid d_D(P_{\Xtraj{k}}, \varrho_X^s) \leq \eps \right\rbrace
            \end{equation*}
            satisfies the inequality $\Dset_{\eps} \geq N - (\delta + C_c)/\alpha_3(\eps)$.
            \label{defn:TurnpikeDistr}
            \item \textbf{$r$-th moment turnpike property} for some $r \leq n$ if \JS{there exists $\alpha_4 \in \K_{\infty}$ such that the following holds: For each $X_0 \in \XX$ there exists a $C_d > 0$} 
            such that for each $\delta > 0$ and $N \in \N$, each control sequence $\mathbf{U} \in \FF^N(X_0)$ satisfying the condition $J_N(X_0,\mathbf{U}) \leq \delta + N \ell(\mathbf{X}^s,\mathbf{U}^s)$ and each $\eps > 0$ the value 
            \begin{equation*}
            \begin{split}
                \Mset^r_{\eps} &:= \# \Big\lbrace k \in \mathcal{N} \mid \\
                &\Big \vert \E[\normX{ \Xtraj{k} }^r]^{1/r} - \E[\normX{ \XstrajShort{k} }^r]^{1/r} \Big \vert \leq \eps \Big\rbrace
            \end{split}
            \end{equation*}
            satisfies the inequality $\Mset^r_{\eps} \geq N - (\delta + C_d)/\alpha_4(\eps)$. \label{defn:TurnpikeMoments}
        \end{enumerate}
    \end{defn}
    
    The turnpike behaviors given in the above definition state that those solutions to the stochastic optimal control problem \eqref{eq:stochOCP}, whose cost is close to that of the stationary process $(\mathbf{X}^s,\mathbf{U}^s)$, most of the time stay close to the stationary state process $\mathbf{X}^s$. The difference of the four turnpike behaviors \ref{defn:TurnpikeLp})--\ref{defn:TurnpikeMoments}) in Definition~\ref{defn:stochTurnpike} lies in the way how ``close'' is defined, i.e., in how we measure the distance between the solutions and the stationary state process. 
    
    In \ref{defn:TurnpikeLp}), we use the $L^r$-norm to define the neighborhood of the stationary process, while in \ref{defn:TurnpikeProb}), this neighborhood is defined in an appropriate in-probability sense for the solutions paths. 
    The advantage of the in-probability definition is that it gives us a more descriptive interpretation of the pathwise solution behavior. 
    It says that the probability that a single realization path does not have the turnpike property in the deterministic sense is low. 
    This pathwise property of a realization is easy to be recognized in numerical simulations, without the need to know the stationary process, which is often difficult to calculate in applications. 
    Specifically, in a preceding simulation study ~\cite{Ou2021} we have observed this pathwise property. 
    In particular, for different initial conditions but identical disturbance realization paths, we observed that the middle part of finite-horizon optimal trajectories behaved very \JS{similar} as all solutions were close to the \textit{pathwise-in-probability turnpike}.
    
    In part \ref{defn:TurnpikeDistr}) of Definition~\ref{defn:stochTurnpike}, we focus on the underlying distributions of the random variables and say that a solution is near to $\mathbf{X}^s$ if their distributions are close to each other.
    Although the norm used for random variables is defined by the selection of the Banach-space $\X$, the norm used to measure the distance between distributions is still arbitrary in our setting, and several choices can lead to stronger or weaker statements about the behavior of the distributions. We specify two special cases of the turnpike and dissipativity in distribution. 
    \JS{For more details on the metrics used in the definition below we refer to \cite{Villani2009,Rachev1991}.}
    
    \begin{defn}\label{def:wasserstein_weak}
    \begin{enumerate}
        \item We say that the problem is (strictly) Wasserstein dissipative or has the Wasserstein turnpike property of order $r$, respectively, for some $r \leq n$ if Definition \ref{defn:DissiDistr} or Definition \ref{defn:stochTurnpike}(c), respectively, holds with $d_D$ being the Wasserstein distance of order $r$, i.e., 
    \begin{equation*}
    \begin{split}
        &d_D(P_X,P_Y) = d_{W_r}(P_X, P_Y)  \\
        &:= \inf \left\lbrace \left(\Exp{\normX{ \bar{X} - \bar{Y} }^r} \right)^{1/r}: ~ \bar{X} \sim P_X,~\bar{Y} \sim P_Y \right\rbrace.
    \end{split}
    \end{equation*} \label{defn:WassersteinTurnpike}
    \item We say that the problem is (strictly) weakly distributionally dissipative or has the weak distributional turnpike property, respectively, if Definition \ref{defn:DissiDistr} or Definition \ref{defn:stochTurnpike}(c), respectively, holds with $d_D$ being the Lévy–Prokhorov metric, i.e.,
    \begin{equation*}
    \begin{split}
        d_D(&P_X,P_Y) =  d_{LP}(P_X, P_Y ) \\
        &:= \inf \left\lbrace d_{KF} (\bar{X}, \bar{Y}) :~ \bar{X} \sim P_X,~\bar{Y} \sim P_Y \right\rbrace,
    \end{split}
    \end{equation*}
    where 
    \begin{equation*}
        d_{KF}(X,Y):= \inf \left\lbrace \eps > 0 : \Prob \left( \normX{ X - Y } > \eps \right) \leq \eps \right\rbrace
    \end{equation*}
    denotes the Ky-Fan-metric for two random variables $X,Y \in \Llp[n]{\X}$. \label{defn:WeakTurnpike}
    \end{enumerate}
    \end{defn}
    Here, the term ``weak'' in \ref{defn:WeakTurnpike}) refers to the fact that Lévy–Prokhorov metric can be used to characterize the weak convergence of measures.

    The reason why these metrics are of particular interest to us will become clear when we compare them to the distance measures defining our turnpike properties regarding the random variables in Section \ref{sec:diss_imp_Lr}. We can directly observe that the Wasserstein distance has a natural connection to $L^r$ norm while the Key-Fan metric is close to how we define the neighborhood of $\mathbf{X}^s$ for the pathwise-in-probability turnpike property.
    Finally, part \ref{defn:TurnpikeMoments}) of Definition \ref{defn:stochTurnpike} formalizes that the $r$-th moments of the respective solutions are close to each other most of the time.

    \begin{rem}
        In general it is possible that there is no control $U \in \UU^N$ satisfying the performance bound $J_N(X_0,\mathbf{U}) \leq \delta + N \ell(\mathbf{X}^s,\mathbf{U}^s)$ demanded in Definition~\ref{defn:stochTurnpike}. 
        Thus, to guarantee that there are such control sequences one has to impose suitable controllability or reachability assumptions on the stochastic system, similar to, e.g., the analysis in \cite[Section~6]{Gruene2013} in the deterministic case. 
        However, the technical details of such assumptions are beyond the scope of this paper.
    \end{rem}

    \begin{rem}
        Since our definitions for turnpike and dissipativity are concerned with measuring the distance between state processes, one could also refer to these properties as \emph{state} turnpike or \emph{state} dissipativity, respectively. 
        Another approach to gain additional information about the behavior of the control processes is to measure also the distances of the controls or policies in these definitions, which would lead to the definition of \emph{input-state} turnpike and \emph{input-state} dissipativity.
        These stronger properties \JS{hold, for example,} in the standard stochastic linear-quadratic setting, cf.\ \cite{CDCPaper}. 
        \JS{In addition, one could also take the joint distributions of the state-control pairs into account, by defining an appropriate distance measure.}
        \JS{W}e conjecture that one could extend the analysis in this paper to \JS{these forms of} input-state turnpike properties \JS{by} using suitable {metrics and norms for the corresponding quantities and imposing additional} regularity assumptions on the stochastic problem. 
        \JS{In particular, we expect that the main results of this paper also hold in this setting and that we can extend the characterization from Theorem~\ref{thm:stationaryOpti} such that not only the optimal state distribution but also the optimal control policy or distribution is unique, assuming an appropriate form of \emph{input-state} dissipativity}
        \JS{However,} in order to keep this paper concise we leave \JS{the details of these extensions as questions} for future research.
    \end{rem}
    
    \JS{\begin{rem}
        An important question connected to the turnpike properties given in Definition~\ref{defn:stochTurnpike} is how to compute the distances between the solutions and the stationary process appearing in this definition. 
        While for the distributional turnpike and the $r$-th moment turnpike property it is sufficient to know the distributions (or single moments) of the solutions at each time $k$, the $L^r$ turnpike and the pathwise-in-probability turnpike property takes also the joint distribution of the solution $X(k)$ and the stationary process $X^s(k)$ into account. 
        Note that, in general, knowing the distributions of $X(k)$ and $X^s(k)$ is not sufficient to compute their joint distribution, unless the two states are stochastically independent. However, since the two processes are computed with respect to the same disturbance sequence $\mathbf{W}$ this is usually not the case for $k > 0$. 
        Nevertheless, the evaluation of the distance measures with respect to the joint distributions is still possible by forward propagation of the entire random variables or by pathwise Monte Carlo sampling as in \cite[Section~6]{CDCExtension}.
    \end{rem}}

    \section{Main results} \label{sec:MainResults}

    In this section, we investigate the relationship between the different turnpike and dissipativity notions.
    Furthermore, we introduce the concept of optimal stationary solutions replacing the optimal deterministic steady state in stochastic systems. 

    \subsection{Implications of strict $\mathit{L^r}$ dissipativity} \label{sec:diss_imp_Lr}
    
    We start with the following theorem, which shows that $L^r$ dissipativity implies the turnpike property in the same sense. Its proof uses the same arguments as in deterministic settings, cf.\ \cite{Gruene2013, Gruene2018a}.
    \JS{In particular, analogous to the deterministic case the proof uses the concept of rotated stage costs. These stage costs are defined as  
    \begin{equation} \label{eq:rot_stagecosts}
    \begin{split}
        \tilde{\ell}(\XtrajShort{k},U(k))
        &:= \; \ell(\XtrajShort{k},U(k)) - \ell(\mathbf{X}^s,\mathbf{U}^s) \\
            &\quad + \lambda(k,\XtrajShort{k}) - \lambda(k+1,\XtrajShort{k+1}),
    \end{split}
    \end{equation}
    and, thus, collect the terms on the left-hand side of the dissipation inequality \eqref{eq:DissiIneqLp}. 
    }
    
    \begin{thm} \label{thm:DissiTurnpikeLp}
        Assume that the stochastic optimal control problem \eqref{eq:stochOCP} is strictly $L^r$ dissipative. Then it has the $L^s$ turnpike property for all $\JS{1\leq s \leq r}$.
    \end{thm}
    \begin{proof}
        \JS{Choose $C_a > \lambda(0,X_0) - M \geq 0$}
        where $M \in \R$ is a uniform lower bound on $\lambda$ from Definition \ref{defn:DissiLp} for all $k \in \N_0$.
        Then for all $\mathbf{U} \in \FF^N(X_0)$ with $J_N(X_0,\mathbf{U}) < \delta + N\ell(\mathbf{X}^s,\mathbf{U}^s)$ we get 
        \begin{equation} \label{eq:turnpikeContr}
            \begin{split}
                \tilde{J}_N(X_0,\mathbf{U}) \JS{:=}&  \sum_{k=0}^{N-1} \ell(\XtrajShort{k},U(k)) - \JS{\ell(\mathbf{X}^s,\mathbf{U}^s)} \\
                &\quad + \lambda(k,\XtrajShort{k}) - \lambda(k+1,\XtrajShort{k+1})) \\
                =& \; J_N(X_0,\mathbf{U}) - N \ell(\mathbf{X}^s,\mathbf{U}^s) \\ 
                & + \lambda(0,\XtrajShort{0}) - \lambda(N,\XtrajShort{N}) \leq \delta + C_a.
            \end{split}
        \end{equation}
        Now assume that $\Lset^r_{\eps} < N-(\delta+C_a)/\alpha_1(\eps)$ with $\alpha_1 = \alpha \in \K_{\infty}$ from Definition~\ref{defn:DissiLp}. Then there is a set $\mathcal{M} \subset \lbrace 0,\ldots,N-1 \rbrace$ of $N-\Lset^r_{\eps} > (\delta+C_a)/\alpha_1(\eps)$ time instants such that 
        $\E [ \Vert \XtrajShort{k} - \XstrajShort{k} \Vert^r ] \geq \eps$ for all $k \in \mathcal{M}$. Using inequality \eqref{eq:DissiIneqLp} from Definition~\ref{defn:DissiLp}, this implies 
        \begin{equation*}
            \begin{split}
                \tilde{J}_N(X_0,\mathbf{U}) \geq& \sum_{k=0}^{N-1} \alpha_1\left(\E\big[ \normX{ \XtrajShort{k} - \XstrajShort{k} }^r\big]\right) \\
                \geq& \sum_{k \in \mathcal{M}} \alpha_1\left(\E\big[  \normX{ \XtrajShort{k} - \XstrajShort{k} }^r \big]\right) \\
                >& \dfrac{\delta+C_a}{\alpha_1(\eps)} \alpha_1(\eps) = \delta + C_a
            \end{split}
        \end{equation*}
        which contradicts \eqref{eq:turnpikeContr} and, thus, proves the theorem for $s=r$. For $s < r$ the assertion follows directly from $\Vert X - Y \Vert_{L^s} \leq \Vert X - Y \Vert_{L^r}$ for all $X,Y \in \Llp[n]{\X}$.
    \end{proof}

    The following theorem shows that one can also obtain the pathwise-in-probability turnpike property from the $L^r$ turnpike property and, thus, from the strict $L^r$ dissipativity.

    \begin{thm} \label{thm:TurnpikeLpProb}
        Assume that the stochastic optimal control problem \eqref{eq:stochOCP} has the $L^r$ turnpike property. Then it also has the pathwise-in-probability turnpike property.
    \end{thm}
    \begin{proof}
        Using the Markov inequality, we get 
        \begin{equation} \label{eq:MarkovIneq}
        \begin{split}
            \Prob\big(\normX{ \XtrajShort{k} - \XstrajShort{k} } \geq \eps \big) \leq \dfrac{1}{\eps^r}\Exp{\normX{ \XtrajShort{k} - \XstrajShort{k} }^r}
        \end{split}
        \end{equation}
        for all $r \in \N$.
        Further, by the definition of the turnpike $L^r$ property, cf.\ Definition~\ref{defn:stochTurnpike}, we know that for the $\K_{\infty}$ function $\alpha_2(\eps) := \alpha_1(\eps^{r+1})$ there are at least $N-(\delta+C_a)/\alpha_1(\eps^{r+1})$ time instants for which $\E [ \normX{ \XtrajShort{k} - \XstrajShort{k} }^r ] \leq \eps^{r+1}$.
        Using equation \eqref{eq:MarkovIneq}, this gives 
        $\Prob\left( \normX{ \XtrajShort{k} - \XstrajShort{k} } \geq \eps \right) \leq \eps$
        for all these time instants which implies 
        $\Prob\left( \normX{ \XtrajShort{k} - \XstrajShort{k} } \leq \eps \right) \geq 1-\eps$
        and, thus, proves the claim with $C_b = C_a$.
    \end{proof}

    Next, we show that pathwise turnpike in probability, in turn, implies a turnpike property for the underlying distributions. However, in general, this turnpike property cannot be concluded for an arbitrary metric on $\PP_n(\X)$. It does, however, hold for the Lévy–Prokhorov metric, as the following theorem shows.

    \begin{thm} \label{thm:TurnpikeProbDistr}
        If the stochastic optimal control problem \eqref{eq:stochOCP} has the pathwise-in-probability turnpike property then it also has the weak distributional turnpike property in the sense of part \ref{defn:WeakTurnpike}) of Definition \ref{def:wasserstein_weak}.
    \end{thm}
    \begin{proof}
        Since the optimal control problem has the pathwise-in-probability turnpike property, we know by Definition~\ref{defn:stochTurnpike} that there is an $\alpha_2 \in \K_{\infty}$ such that there are at least $N - (\delta+C_a)/\alpha_2(\eps)$ time instants for which $\Prob\left(\normX{ \XtrajShort{k} - \XstrajShort{k} } \leq \eps \right) \geq 1-\eps$ holds. Hence, it follows that 
        \begin{equation*}
        \begin{split}
            d_{KF}& (\XtrajShort{k},\XstrajShort{k}) \\
            &= \inf_{\epsilon > 0} \left \{ \Prob (\normX{ \XtrajShort{k} - \XstrajShort{k} } > \epsilon) \leq \epsilon \right \} \leq \eps
        \end{split}
        \end{equation*}
        holds for all these time instants, and thus, also
        \begin{equation*}
        \begin{split}
            d_{LP} &(P_{\XtrajShort{k}}, \varrho^s_X) \\
            &= \inf \{ d_{KF} ( X, Y ) : X \sim P_{\XtrajShort{k}}, ~ Y \sim \varrho^s_X \} \leq \eps,
        \end{split}
        \end{equation*}
        which proves the claim with $\alpha_3 = \alpha_2 \in \K_{\infty}$ and $C_c = C_b$.
    \end{proof}

    Similar to stochastic convergence properties, the distributional turnpike property with respect to the Lévy–Prokhorov metric is a relatively weak property since, in general, we can not even conclude statements about the behavior of the moments from it. However, the Wasserstein metric is strong enough to let us make implications about the moments. Unfortunately, from the pathwise turnpike in probability, we cannot conclude the Wasserstein turnpike property, but this conclusion holds for the $L^r$ turnpike property, as the following theorem shows.

    \begin{thm} \label{thm:TurnpikeLpWp}
        Assume the stochastic optimal control problem \eqref{eq:stochOCP} has the $L^r$ turnpike property. Then it has the Wasserstein turnpike property of order $s$ for all $s \leq r$ in the sense of part \ref{defn:WassersteinTurnpike}) of Definition \ref{def:wasserstein_weak}.
    \end{thm}
    \begin{proof}
         Since $r \leq n$, we know that  
         \begin{equation*}
         \begin{split}
             \Vert X - &Y \Vert_{L^r} \geq \\
             &\inf \left\lbrace \left(\Exp{\normX{ \bar{X} - \bar{Y} }^r} \right)^{1/r}: ~ \bar{X} \sim P_X,~\bar{Y} \sim P_Y \right\rbrace \\
             &= d_{W_r}(P_X,P_Y) \geq 0
         \end{split}
         \end{equation*}
         holds for all $X,Y \JS{\in} \Llp[n]{\X}$. Thus, 
         \begin{equation*}
         \begin{split}
            \Vert \XtrajShort{k} - &\XstrajShort{k} \Vert_{L^r}^r = \Exp{\normX{ \XtrajShort{k} - \XstrajShort{k} }^r} \leq \eps
         \end{split} 
         \end{equation*}
         implies that $d_{W_r} ( P_{\XtrajShort{k}} , \varrho^s_X ) \leq \eps^{1/r}$ for all $\eps > 0$, which proves the claim for $s = r$ with $C_c = C_a$ and $\alpha_3(\eps) = \alpha_1(\eps^{1/r}) \in \K_{\infty}$. For $s < r$ the assertion follows directly from $\Vert P_X - P_Y \Vert_{\JS{L^s}} \leq \Vert P_X - P_Y \Vert_{\JS{L^r}}$ for all $X,Y \in \Llp[n]{\X}$, cf.\ \cite[Remark~6.6]{Villani2009_book}.
    \end{proof}

    The next result shows that the Wasserstein turnpike property is indeed strong enough to imply turnpike of different moments.

    \begin{thm} \label{thm:TurnpikeWpMoments}
        \JS{Assume} the stochastic optimal control problem \eqref{eq:stochOCP} has the Wasserstein turnpike property \JS{of order $r$}. Then it has the $s$-th moment turnpike property for all $s \leq r$.
    \end{thm}
    \begin{proof}
        Because $d_{W_s}(\cdot,\cdot)$ defines a metric on $\PP_n(\X)$, we get by the triangle inequality
        \begin{equation*}
          d_{W_s} (P_X , 0 )\leq d_{W_s} ( P_X , P_Y ) + d_{W_s} (P_Y , 0 )
        \end{equation*}
        and 
        \begin{equation*}
          d_{W_s} ( P_Y , 0 ) \leq d_{W_s} (P_X , P_Y ) + d_{W_s}  (P_X , 0 )
        \end{equation*}
        which implies 
        \begin{equation*} 
        \begin{split}
            \vert d_{W_s} (P_X , 0 ) - &d_{W_s} (P_Y , 0 ) \vert \\
            &\leq d_{W_s} (P_X, P_Y ) \leq d_{W_r} ( P_X , P_Y )
        \end{split}
        \end{equation*}
        for all $X,Y \in \Llp[n]{\X}$ and $s \leq r \leq n$. Further, we have 
        \begin{equation*}
        \begin{split}
            d_{W_s} (P_X , 0 ) &= \inf \left\lbrace \left(\Exp{\normX{ \bar{X} - 0 }^s} \right)^{1/s}: ~ \bar{X} \sim P_X \right\rbrace \\
            &=  \Exp{\normX{ X }^s}^{1/s}
        \end{split}
        \end{equation*}
        and, thus, $d_{W_r} ( P_{\XtrajShort{k}} , \varrho_X^s) \leq \eps$ implies $\vert \Exp{\normX{ \XtrajShort{k} }^s}^{1/s} - \Exp{\normX{ \XstrajShort{k} }^s}^{1/s} \vert \leq \eps$ which proves the theorem with $C_d = C_c$ and $\alpha_4 = \alpha_3 \in \K_{\infty}$. 
    \end{proof}

    \subsection{Optimal stationary solutions}
    One of the key results in deterministic turnpike theory is that the steady state at which the turnpike property occurs---sometimes also referred to as {\em the} turnpike---is an optimal steady state. The generalization of this characterization to our setting is done in the following theorem.
    \JS{Note that} it is sufficient to prove this statement for the stationary distribution in the sense of Definition~\ref{defn:stationaryDistr}, as \JS{equation \eqref{eq:equiDistrStat}} then implies that it also holds for Definition~\ref{defn:stationaryProcess}.

    \begin{thm} \label{thm:stationaryOpti}
        Let the stochastic optimal control problem \eqref{eq:stochOCP} be distributionally dissipative at $(\varrho_X^s,\pi^s)$ with respect to a metric $d_D$. Then the stationary distribution $\varrho_X^s$ and the stationary feedback $\pi^s$ are an optimal solution to 
        \begin{equation} \label{eq:stationaryOpti}
        \begin{split}
            &\min_{\pi,\varrho_X} \ell(X,U) \\
            s.t.& ~ X \sim \varrho_X, ~ U = \pi(X), ~ \\
            &\varrho_X = \mathcal{T}(\varrho_X,\pi), ~ (X,\pi(X)) \in \Y.
        \end{split}
        \end{equation}
        Moreover, if the problem is strictly dissipative in distribution, then the stationary distribution $\varrho_X^s \in \Gamma(\X)$ is the unique (partial) solution of this problem, i.e., for every other solution $(\tilde{\varrho}_X^s, \tilde{\pi}^s) \in \Gamma(\X) \times \Pi_0(\X,\U)$ we get $d_D( \varrho_X^s, \tilde{\varrho}_X^s) = 0$.
    \end{thm}
    \begin{proof}
        We prove the claim by contradiction. Assume distributional dissipativity and that there is a solution $(\tilde{\varrho}_X^s, \tilde{\pi}^s) \in \Gamma(\X) \times \Pi_0(\X,\U)$ to problem \eqref{eq:stationaryOpti} such that 
        $\ell(\tilde{X}^s,\tilde{\pi}^s(\tilde{X}^s)) < \ell(X^s,\pi^s(X^s))$
        for $\tilde{X}^s \sim \tilde{\varrho}_X^s$, $X^s \sim \varrho_X^s$. Then we get
        \begin{equation*}
        \begin{split}
            &\ell(\tilde{X}^s,\tilde{\pi}^s(\tilde{X}^s)) - \ell(X^s,\pi^s(X^s)) + \Lambda(P_{\tilde{X}^s}) - \Lambda(P_{(\tilde{X}^s)^+}) \\ 
            &\quad = \ell(\tilde{X}^s,\tilde{\pi}^s(\tilde{X}^s)) - \ell(X^s,\pi^s(X^s)) < 0,
        \end{split}
        \end{equation*}
        which contradicts \eqref{eq:DissiIneqDistr} with $\alpha \equiv 0$ and, thus, the dissipativity of the optimal control problem. Here, we used that $P_{\tilde{X}^s} = P_{(\tilde{X}^s)^+}$ for $\tilde{X}^s \sim \tilde{\varrho}_X^s$ due to the stationarity condition. 
        
        Now consider that the problem is strictly distributionally dissipative and $\ell(\tilde{X}^s,\tilde{\pi}^s(\tilde{X}^s)) = \ell(X^s,\pi^s(X^s))$, but $d_D(\varrho_X^s, \tilde{\varrho}_X^s) > 0$. Then we get by the definition of strict distributional dissipativity that there is a function $\alpha \in \K_{\infty}$ such that
        \begin{equation*}
        \begin{split}
            0 = &\ell(\tilde{X}^s,\tilde{\pi}^s(\tilde{X}^s)) - \ell(X^s,\pi^s(X^s)) \\
            &+ \Lambda(P_{\tilde{X}^s}) - \Lambda(P_{(\tilde{X}^s)^+}) \geq \alpha(d_D( \varrho_X^s , \tilde{\varrho}_X^s)),
        \end{split}
        \end{equation*}
        which is again a contradiction since $\alpha(d_D(\varrho_X^s, \tilde{\varrho}_X^s)) > 0$ for $d_D(\varrho_X^s, \tilde{\varrho}_X^s) > 0$.
    \end{proof}

    Theorem~\ref{thm:stationaryOpti} shows how we can uniquely characterize an optimal stationary distribution similar to the \JS{characterization of an optimal equilibrium in the} deterministic setting.
    \JS{This is important since it shows that the long-time behavior of the solutions of problem \eqref{eq:stochOCP} is determined not by an arbitrary stationary solution but rather by the {\em optimal} stationary solution. Moreover, the optimality of this stationary solution can be used as a tool to prove that stochastic MPC schemes approximate optimal solutions of infinite horizon stochastic optimal control problems, cf.\ \cite{PerformanceCDC}.} \JS{Theorem \ref{thm:stationaryOpti} also characterizes an optimal $\pi$ for the infinite-horizon optimal control problem. However, in general this feedback is only characterized for the stationary process and not for all optimal solutions, because Remark \ref{rem:opt_pi} holds accordingly for the infinite-horizon problem.}
    \JS{
    \begin{rem}[\JS{Stationarity and the pathwise perspective}]
        Although Theorem~\ref{thm:stationaryOpti} gives us a useful characterization of stationary distributions via an optimality criterion, the problem~\eqref{eq:stationaryOpti} is in general infinite dimensional and thus hard to solve. However, this, in turn, shows us an advantage of the pathwise view: Unlike the other turnpike properties, the pathwise-in-probability turnpike can be observed numerically without the necessity to compute of the stationary process or its distribution. This is possible because we can simulate the realizations of the solutions to \eqref{eq:stochOCP} for a fixed realization of the disturbance $\mathbf{W}$ and check by means of these simulations whether the realizations are close to each other most of the time for sufficiently large optimization horizons $N$. \JS{Notice that path-wise results correspond nicely \JS{to} applications, as real-world measurements are \JS{always taken from} realization paths.}
    \end{rem}
    }
    Although Theorem~\ref{thm:stationaryOpti} considers strict distributional dissipativity, the same statement also holds in the case of strict $L^r$ dissipativity. We could show this by slightly modifying the proof of Theorem~\ref{thm:stationaryOpti} as in \cite[Theorem~5.3]{CDCExtension}, which proves the statement in a linear-quadratic setting. 
    However, we will choose another direction and conclude this fact from more general results on the relation of the two different (strict) dissipativity notions in random variables and in distribution (cf.\ Definitions \ref{defn:DissiDistr} and \ref{defn:DissiLp}). 
    These results provide an interesting contribution in their own right.

    \subsection{Relation between (strict) dissipativity concepts}
    We start this section by proving that strict $L^r$ dissipativity implies strict distributional dissipativity. To this end, we use the following proposition, a variant of \cite[Proposition 3.3]{Gruene2016}  based on \cite[Theorem 1]{willems1972a} in our stochastic setting.

    \begin{prop}[Available storage in distribution] \label{prop:availableStorage}
        Let $\pi^s$ be a stationary feedback of system \eqref{eq:stochSys} with stationary distribution $\varrho_X^s$ such that $(X^s,\pi^s(X^s)) \in \Y$ for $X^s \sim \varrho^s_X$. Then there exists $\alpha \in \K_{\infty}$ (or $\alpha \equiv 0$, respectively) with
        \begin{equation} \label{eq:availableStorage}
        \begin{split}
            &\Lambda(P_{X_0}) := \sup_{\underset{\boldsymbol{\pi} \in \St^N(P_{X_0})}{N \in \N_0}} \sum_{k=0}^{N-1} - \Big( \hat{\ell}(P_{X(k)}, \pi_k) - \hat{\ell}(\varrho^s_X, \pi^s) \\
            &\qquad \qquad \qquad \qquad \qquad - \alpha ( d_D( P_{X(k)}, \varrho_X^s) ) \Big) < \infty
        \end{split}
        \end{equation}
        for all $P_{X_0} \in P_\XX$ if and only if the stochastic optimal control problem \eqref{eq:stochOCP} is strictly distributionally dissipative (or distributionally dissipative, respectively) with respect to the metric $d_D$.
        In this case, inequality \eqref{eq:DissiIneqLp} holds with $\Lambda$ and $\alpha$ from \eqref{eq:availableStorage} and $\Lambda$ is called the {\em available storage}.
    \end{prop}
    \begin{proof}
        Consider the Markov decision process \eqref{eq:stochOCPmeasure}. Then the statement follows directly by \cite[Proposition 3.3]{Gruene2016}, which proves the claim since distributional dissipativity of the probability \JS{measure-based} problem \eqref{eq:stochOCPmeasure} and the stochastic optimal control problem \eqref{eq:stochOCP} are equivalent.
    \end{proof}

    Now, we can prove that strict $L^r$ dissipativity implies strict distributional dissipativity by showing that the function $\Lambda$ from \eqref{eq:availableStorage} is bounded.

    \begin{thm} \label{thm:DissiLrToDissiDistr}
        Assume the stochastic optimal control problem \eqref{eq:stochOCP} is strictly $L^r$ dissipative. Then it is also strictly weakly distributionally dissipative and strictly Wasserstein dissipative of order $s$ for all $s \leq r$.
    \end{thm}
    \begin{proof}
        Let $d_D(\cdot,\cdot) = d_{W_s}(\cdot,\cdot)$.
        From Definition~\ref{defn:DissiLp} we know that there is a function $\hat{\alpha} \in \K_{\infty}$ such that
        \begin{equation*}
        \begin{split}
            &\ell(\XtrajShort{k},U(k)) - \ell(\mathbf{X}^s,\mathbf{U}^s) \\
            &\qquad \quad + \lambda(k,\XtrajShort{k}) - \lambda(k+1,\XtrajShort{k+1}) \\ 
            &\geq \hat{\alpha}(\Vert \XtrajShort{k} - \XstrajShort{k} \Vert_{L^r}^r) =  \alpha(\Vert \XtrajShort{k} - \XstrajShort{k} \Vert_{L^r})
        \end{split}
        \end{equation*}
        for $\alpha(z) = \hat{\alpha}(z^r) \in \K_{\infty}$. 
        Now let $M$ be a uniform lower bound on $\lambda$.
        Then for $(\varrho_X^s,\pi^s) \sim (\mathbf{X}^s, \mathbf{U}^s)$ we get
        \begin{equation} \label{eq:inequalityStorage}
            \begin{split}
                \sup_{\underset{\boldsymbol{\pi} \in \St^N(P_{X_0})}{N \in \N_0}} \sum_{k=0}^{N-1} - \Big( \hat{\ell}(P_{\XtrajShort{k}}, \pi_k) &- \hat{\ell}(\varrho^s_X, \pi^s) \\
                - \alpha ( &d_D(P_{\XtrajShort{k}}, \varrho_X^s) ) \Big) \\
                = \sup_{\underset{\mathbf{U} \in \FF^N(X_0)}{N \in \N_0}} \sum_{k=0}^{N-1} - \Big( \ell(\XtrajShort{k}), U(k)) &- \ell(\mathbf{X}^s, \mathbf{U}^s) \\
                 - \alpha ( d_D(P_{\XtrajShort{k}}, &P_{\XstrajShort{k}})) \Big) \\ 
                 \leq \sup_{\underset{\mathbf{U} \in \FF^N(X_0)}{N \in \N_0}} \sum_{k=0}^{N-1} - \Big( \ell(\XtrajShort{k}, U(k)) &- \ell(\mathbf{X}^s, \mathbf{U}^s) \\
                - \alpha ( \Vert \XtrajShort{k}& - \XstrajShort{k} \Vert_{L^r} ) \Big) \\
                \leq \sup_{N \in \N_0, \mathbf{U} \in \FF^N(X_0)} \sum_{k=0}^{N-1} \Big( \lambda(k,\XtrajShort{k}) &\\
                - \lambda(k+1,&\XtrajShort{k+1}) \Big) \\
                \leq \lambda(0,X_0) - M < \infty \qquad \qquad &
            \end{split}
        \end{equation}
        for all $P_{X_0} \in P_{\XX}$ and noise sequences $\mathbf{W}$. Thus, the assertion for strict Wasserstein dissipativity follows by Proposition~\ref{prop:availableStorage}. The assertion for strict weak distributional dissipativity follows analogously, since \JS{the} inequality above also holds for $d_D(\cdot,\cdot) = d_{LP}(\cdot,\cdot)$.
    \end{proof}

    It is easy to see that Theorem~\ref{thm:DissiLrToDissiDistr} also holds for non-strict (meaning ``not necessarily strict'') dissipativity. Indeed, for dissipativity the converse is is also true, as the following lemma shows.

    \begin{lem} \label{lem:DissiLrDissiDistr}
        The stochastic optimal control problem \eqref{eq:stochOCP} is $L^r$ dissipative \JS{for all $1 \leq r \leq n$} if and only if it is distributionally dissipative. 
    \end{lem}
    \begin{proof}
        ``$\Rightarrow$'' Assume that the stochastic optimal control problem \eqref{eq:stochOCP} is strict $L^r$ dissipative. Then the inequality \eqref{eq:inequalityStorage} also holds with $\alpha \equiv 0$ which implies distributional dissiptivity by Proposition~\ref{prop:availableStorage}. \\
        ``$\Leftarrow$'' Assume that the stochastic optimal control problem is distributionally dissipative, i.e.,
        \begin{equation}
            \begin{split} \label{eq:DissiIneqDistr2}
                0 \leq& \ell(X,\pi(X)) - \ell(X^s,\pi^s(X^s)) + \Lambda(P_X) - \Lambda(P_{X^+})
            \end{split}
        \end{equation}
        holds for all for all $(X,\pi(X)) \in \Y$ and $X^s \sim \varrho^s_X$.
        Further, we know that $P_{\XtrajShort{k}} = P_{X(k)}$ for all noise sequences $\mathbf{W}$ and for $(\mathbf{X}^s, \mathbf{U}^s) \sim (\varrho_X^s,\pi^s)$ the identity $\ell(X^s,\pi^s(X^s)) = \ell(\mathbf{X}^s, \mathbf{U}^s)$ holds for $X^s \sim \varrho^s_X$.
        Thus, inequality \eqref{eq:DissiIneqDistr2} directly implies 
        \begin{equation*}
            \begin{split}
                0 \leq \; & \ell(\XtrajShort{k},U(k)) - \ell(\mathbf{X}^s, \mathbf{U}^s) \\
                &+ \lambda(k,\XtrajShort{k}) - \lambda(k+1,\XtrajShort{k+1}) 
            \end{split}
        \end{equation*}
        with $\lambda(k,\XtrajShort{k}) = \Lambda(P_{\XtrajShort{k}})$ for all $k \in \N_0$ and $(\XtrajShort{k}),U(k)) \in \Y$ with $U(k) \in \Llp[l]{\U}$ and $U(k) = \pi(X(k))$ for some $\pi \in \Pi_0(\X,\U)$, which implies $L^r$ dissipativity.
    \end{proof}
    
    The reason why the equivalence of the two dissipativity concepts in the sense of Lemma~\ref{lem:DissiLrDissiDistr} holds is that for dissipativity the metric used to measure the distance between the random variables has no impact on the dissipation inequalities \eqref{eq:DissiIneqDistr} and \eqref{eq:DissiIneqLp} since $\alpha \equiv 0$. 
    However, for \emph{strict} dissipativity the inequalities depend on these metrics and since we cannot bound the $L^r$-norm from above by a metric on the space of probability measures, Lemma~\ref{lem:DissiLrDissiDistr} does not hold for strict dissipativity. 
    \JS{This is illustrated by the following example.
    \begin{bsp} \label{expl:counterexample}
        Let $d_D$ be an arbitrary metric on $\mathcal{P}_n(\mathcal{X})$ and consider the optimal control problem
        \begin{equation} \label{eq:counterexample}
            \begin{split}
                &\min_{\mathbf{U}} J_N(X_0,\mathbf{U}) = \sum_{k=0}^{N-1} d_D(P_{X(k)},\varrho^*)  \\
                &\quad s.t. ~  X(k+1) = X(k) + U(k)W(k), ~ X(0) = X_0 \\
                & \qquad  X(k) \in \Llp[n]{\R}, ~ U(k) \in \Lp[n]{\R} \\ 
            \end{split}
        \end{equation}
        with $\varrho^* \sim \Normal(0,\Sigma)$, $\Sigma \in \R$. It is easy to see that the problem \eqref{eq:counterexample} is strictly distributional dissipative at $(\varrho^*,0)$ with respect to the metric $d_D$ and storage function $\Lambda \equiv 0$. Moreover, any pair $(\mathbf{X}^s,\mathbf{U}^s)$ with $\mathbf{X}^s(k) \sim \Normal(0,\Sigma)$ and $\mathbf{U}^s \equiv 0$ is stationary in the sense of Definition~\ref{defn:stationaryProcess}, and thus, in particular for every such stationary pair $(\mathbf{X}^s,\mathbf{U}^s)$ the pair $(-\mathbf{X}^s,\mathbf{U}^s)$ is also stationary since $\varrho^* \sim \Normal(0,\Sigma)$ is a symmetric distribution around zero. Hence, following the arguments from \cite[Lemma~5.4]{CDCExtension}, problem \eqref{eq:counterexample} cannot be strictly dissipative in $L^r$ for all $r \leq n$ since $\Vert X^s(k) - (-X^s(k)) \Vert_{L^r} \nrightarrow 0$ as $k \rightarrow \infty$.
    \end{bsp}}
    
    \JS{
    Note that the implications from Example~\ref{expl:counterexample} also hold if we replace the metric $d_D$ by a dissimilarity measure in the sense of \cite{Gros2022}. This shows that also for dissimilarity measures Lemma~\ref{lem:DissiLrDissiDistr} does not hold in the strict case.
    }

    \subsection{Implications of strict distributional dissipativity}
    Now that we know how to switch between the different settings, we conclude the presentation of our main results by the following theorem, which shows that strict distributional dissipativity implies distributional turnpike. Again, the proof uses the same arguments as in deterministic settings. 
    \JS{Note that the rotated costs used in the proof of Theorem~\ref{thm:DissiTurnpikeDistr} are constructed in the same way as in \eqref{eq:rot_stagecosts} using the storage function $\Lambda$ from the distributional dissipativity inequality \eqref{eq:DissiIneqDistr} instead of the $L^r$-storage function $\lambda$ from Definition~\ref{defn:DissiLp}.}
    
    \begin{thm} \label{thm:DissiTurnpikeDistr}
        Assume that the stochastic optimal control problem \eqref{eq:stochOCP} is strictly distributionally dissipative with respect to a metric $d_D$. Then it has the distributional turnpike property with respect to the same metric. 
    \end{thm}
    \begin{proof}
        \JS{Choose $C_c > \lambda(P_{X_0}) - M \geq 0$}
        where $M \in \R$ is a bound on $\Lambda$ from Definition \ref{defn:DissiDistr}.
        Then for $(\mathbf{X}^s,\mathbf{U}^s) \sim (\varrho^s_X,\pi^s)$ and all $\mathbf{U} \in \FF^N(X_0)$ with $J_N(X_0,\mathbf{U}) < \delta + N\ell(\mathbf{X}^s,\mathbf{U}^s)$ we get 
        \begin{equation} \label{eq:turnpikeContr2}
            \begin{split}
                \JS{\hat{J}}_N(X_0,\mathbf{U}) \JS{:=}&  \sum_{k=0}^{N-1} \ell(\XtrajShort{k},U(k)) - \JS{\ell(\mathbf{X}^s,\mathbf{U}^s)} \\
                &\qquad \qquad  + \Lambda(P_{\XtrajShort{k}}) - \Lambda(P_{\XtrajShort{k+1}}) \\
                =& J_N(X_0,\mathbf{U}) - N \ell(\mathbf{X}^s,\mathbf{U}^s) \\
                &\qquad + \Lambda(P_{X_0}) - \Lambda(P_{\XtrajShort{N}}) \leq \delta + C_c.
            \end{split}
        \end{equation}
        Now assume that $\Dset_{\eps} < N-(\delta+C_c)/\alpha_3(\eps)$ with $\alpha_3 = \alpha \in \K_{\infty}$ from Definition~\ref{defn:DissiDistr}. Then there is a set $\mathcal{M} \subset \lbrace 0,\ldots,N-1 \rbrace$ of $N-\Dset_{\eps} > (\delta+C_c)/\alpha_3(\eps)$ time instants such that 
        $d_D(P_{X(k)}, \varrho_X^s) \geq \eps$ for all $k \in \mathcal{M}$. Using inequality \eqref{eq:DissiIneqDistr} from Definition~\ref{defn:DissiDistr}, this implies 
        \begin{equation*}
            \begin{split}
                \JS{\hat{J}}_N(X_0,\mathbf{U}) \geq& \sum_{k=0}^{N-1} \alpha_3\left(d_D(P_{\XtrajShort{k}}, \varrho_X^s) \right) \\
                \geq& \sum_{k \in \mathcal{M}} \alpha_3\left(d_D(P_{\XtrajShort{k}}, \varrho_X^s) \right) \\
                >& \dfrac{\delta+C_c}{\alpha_3(\eps)} \alpha_3(\eps) = \delta + C_c
            \end{split}
        \end{equation*}
        which contradicts \eqref{eq:turnpikeContr2} and, thus, proves the theorem.
    \end{proof}

    \JS{Theorem~\ref{thm:DissiTurnpikeDistr} holds for all metrics $d_D$, and thus, also for the Wasserstein metric and the Lévy–Prokhorov metric from Definition~\ref{def:wasserstein_weak} which are under special consideration in this paper.}
    \JS{This and all the other main results of this paper are summarized in Figure~\ref{fig:sketch_of_results}.}

    \begin{figure*}
	\begin{center}
        \resizebox{0.95\textwidth}{!}{%
		\begin{tikzpicture}[squarednode/.style={rectangle, draw=black!60, fill=gray!5, very thick, minimum size=5mm}]

  
			\node (DLr) at (1,0) [squarednode,align=center] {Strict $L^r$\\ dissipativity};
            \node (wDDistr) at (1,-5) [squarednode,align=center] {Strict weak distributional\\ dissipativity};
            \node (WpDDistr) at (1,-7) [squarednode,align=center] {Strict Wasserstein dissipativity \\ of order $r$};

			\node (TLr) at (9,0) [squarednode,align=center] {$L^r$ turnpike \\ property};
			\node (TProb) at (9,-2.5) [squarednode,align=center] {Pathwise-in-probability \\ turnpike property};
			\node (wTDistr) at (9,-5) [squarednode,align=center] {Weak distributional \\ turnpike property};
            \node (WpTDistr) at (13.5,-5) [squarednode,align=center] {Wasserstein turnpike \\ property of order $r$};
            \node (TMoment) at (13.5,-7.5) [squarednode,align=center] {$r$-th moment \\ turnpike property};

            \node () at (5.,0.4) [] {Theorem \ref{thm:DissiTurnpikeLp}}; 
            \node () at (-1.3,-3.2) [align=center] {\JS{Theorem} \\ \JS{\ref{thm:DissiLrToDissiDistr}}}; 
            \node () at (2.0,-2.5) [align=center] {\JS{Theorem} \\ \JS{\ref{thm:DissiLrToDissiDistr}}}; 
            \node () at (5.1,-4.6) [] {Theorem \ref{thm:DissiTurnpikeDistr}}; 
            \node () at (5.1,-6.6) [] {Theorem \ref{thm:DissiTurnpikeDistr}}; 
            \node () at (7.6,-1.1) [] {Theorem \ref{thm:TurnpikeLpProb}}; 
            \node () at (7.6,-3.8) [] {Theorem \ref{thm:TurnpikeProbDistr}}; 
            \node () at (14.5,-6.2) [align=center] {Theorem \\ \ref{thm:TurnpikeWpMoments}}; 
            \node () at (12.5,-0.7) [align=center] {Theorem \\ \ref{thm:TurnpikeLpWp}}; 

			\draw[-implies, double, line width=2pt] (DLr) to (TLr);

			\draw[-implies, double, line width=2pt] (TLr) to (TProb);
			\draw[-implies, double, line width=2pt] (TProb) to (wTDistr);

            \draw[-implies, double, line width=2pt] (wDDistr) to (wTDistr);
            \draw[-implies, double, line width=2pt] (WpDDistr) to [bend right=12] (WpTDistr);

			\draw[-implies, double, line width=2pt] (TLr) to [bend left=35] (WpTDistr);
            \draw[-implies, double, line width=2pt] (DLr) to [bend right=75] (WpDDistr);

            \draw[-implies, double, line width=2pt] (WpTDistr) to (TMoment);

            \draw[-implies, double, line width=2pt] (DLr) to (wDDistr);
		\end{tikzpicture} }
    \caption{Schematic sketch of the main results.} \label{fig:sketch_of_results}
	\end{center}
    \end{figure*}
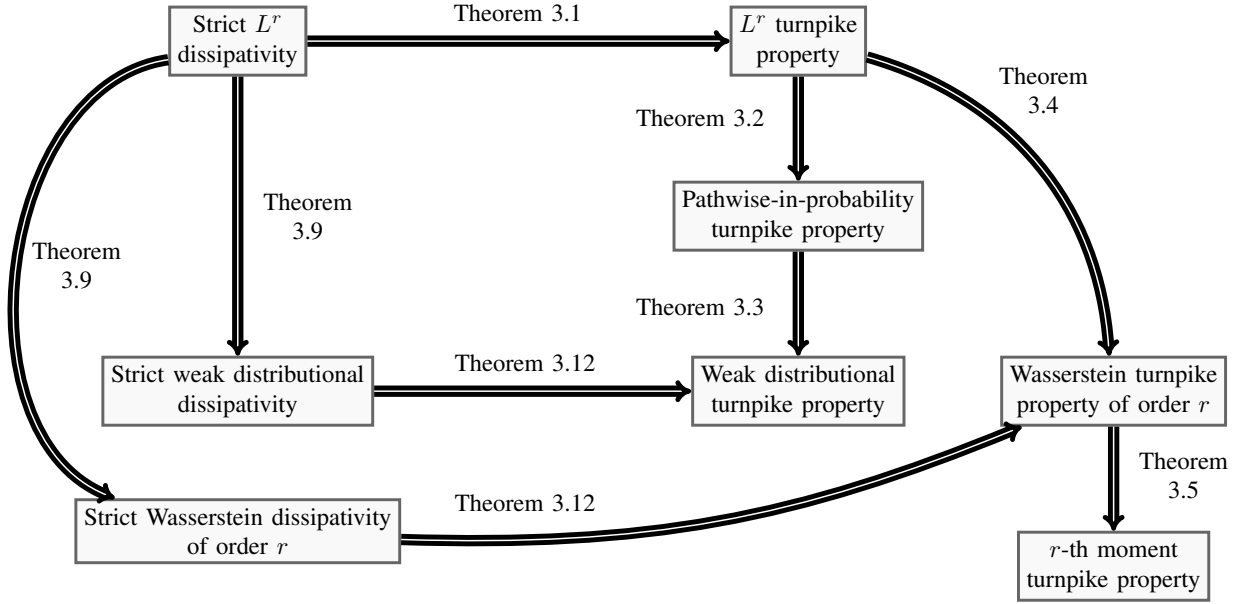

    \section{Example} \label{sec:Example}

    In this section, we illustrate the theoretical results of this paper. For a linear-quadratic example illustrating many of the results in this paper we refer to \cite{CDCExtension}. Here we provide an example showing that strict dissipativity in $L^r$ can also hold for an optimal control problem with nonlinear dynamics and non-quadratic cost. To this end, consider the one-dimensional stochastic optimal control problem 
    \begin{equation} \label{eq:example}
        \begin{split}
            &\min_{\mathbf{U}} J_N(X_0,\mathbf{U}) = \sum_{k=0}^{N-1} \Exp{(U(k) - X(k) )^4} + \gamma \Exp{U(k)}^2  \\
            &\quad s.t. ~  X(k+1) = (U(k)-X(k))^2 + W(k), ~ X(0) = X_0 \\
            & \qquad  X(k) \in \Llp[4]{\R}, ~ U(k) \in \Lp[4]{\R} \\ 
        \end{split}
    \end{equation}
    with $\gamma \geq 0$, $W(k) \sim \varrho_W \in \PP_4(\JS{\R})$, $\Exp{W(k)} = 0$, and the $W(k)$ are stochastically independent pairwise for all $k=0,\ldots,N$ and from $X_0$. Thus, problem \eqref{eq:example} satisfies all assumptions of our setting. 
    
    Furthermore, it is easy to see that the pair $(\mathbf{X}^s,\mathbf{U}^s)$ given by 
    \begin{equation} \label{eq:statProcessesExample}
    \begin{split}
        \XstrajShort{k+1} &= (U^s(k) - \XstrajShort{k})^2 + W(k) = W(k), \\
        U^s(k) &= \XstrajShort{k}
    \end{split}
    \end{equation}
    defines a stationary pair of stochastic processes for problem \eqref{eq:example} with stationary distribution $\varrho_X^s = \varrho_W$. Moreover, we have $\ell(\mathbf{X}^s,\mathbf{U}^s) = 0$ and for $\lambda(k,X(k)) = \Exp{(X(k) - X^s(k))^2}$ we get
    \begin{equation*}
        \begin{split}
            &\ell(\XtrajShort{k},U(k)) - \ell(\mathbf{X}^s,\mathbf{U}^s)  \\
            & \qquad \qquad + \lambda(k,\XtrajShort{k}) -\lambda(k+1,\XtrajShort{k+1}) \\
            =& \Exp{ ( U(k) - \XtrajShort{k} )^4} + \gamma  \Exp{U(k)}^2 + \Exp{(\XtrajShort{k} - \XstrajShort{k})^2} \\
            &\quad - \Exp{((U(k)-\XtrajShort{k}\JS{)}^2 + W(k) - \XstrajShort{k+1})^2} \\
            =& \Exp{ ( U(k) - \XtrajShort{k} )^4} + \gamma  \Exp{U(k)}^2 + \Exp{(\XtrajShort{k} - \XstrajShort{k})^2} \\
            &\quad - \Exp{(\JS{(}U(k)-\XtrajShort{k}\JS{)}^2)^2} \\
            =& \gamma  \Exp{U(k)}^2 + \Exp{(\XtrajShort{k} - \XstrajShort{k})^2} \geq \Vert \XtrajShort{k} - \XstrajShort{k} \Vert_{L^2}^2
        \end{split}
    \end{equation*}
    which shows that problem \eqref{eq:example} is strictly $L^2$ dissipative. 
    Thus, we can conclude that the optimal control problem \eqref{eq:example} is also strictly distributionally dissipative and exhibits the $L^2$, the pathwise-in-probability, the distributional turnpike property, and the turnpike property of the first two moments. 
    While strict dissipativity holds for arbitrary distributions $\varrho_W$ subject to the considered conditions, for our numerical simulation we have chosen $\varrho_W$ to be a two-point distribution such that $W(k)=a :=0.5$ with probability $p_a=0.2$ and $W(k)=b :=-0.125$ with probability $p_b = 0.8$, to facilitate the visualization of the numerical results. 
    Figure~\ref{fig:example} shows all possible realization paths subject to the considered noise of the optimal solutions for $X_0 = 7$, $\gamma=50$ and $N=3,5,10$ together with the optimal stationary process from \eqref{eq:statProcessesExample}\JS{.
    We can clearly observe that for increasing $N$ the realizations (in red, green, blue, orange, purple) spend more and more time in decreasing neighborhoods of the optimal stationary solution (in grey) if we increase the time horizon $N$. This is precisely what the pathwise turnpike property from Definition~\ref{defn:stochTurnpike} describes and, thus, the numerical simulations illustrate our theoretical findings.}
    
    \begin{figure}
        \centering
        \includegraphics[width=0.5\textwidth]{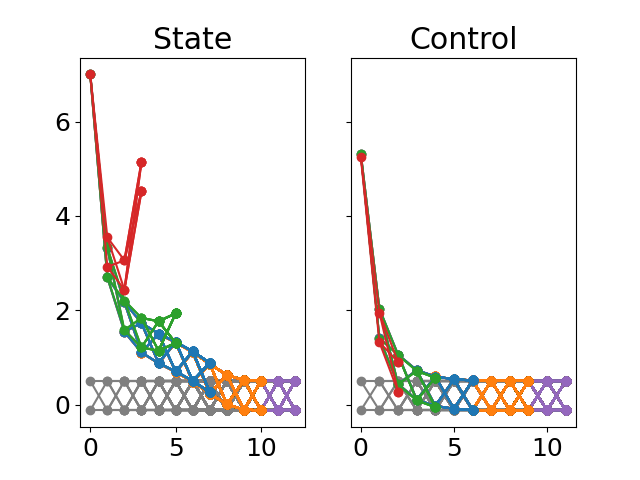}
        \caption{Evolution of all realization paths for the optimal trajectories (left) and controls (right) with $N=3$ in red, $N=5$ in green\JS{, $N=7$ in blue, $N=10$ in orange, and $N=12$ in purple}, as well as for the optimal stationary process (grey)}
        \label{fig:example}
    \end{figure}

    \section{Conclusion} \label{sec:Conclusion}
    In this paper we have shown that strict $L^r$ dissipativity implies various types of turnpike properties for stochastic optimal control problems. 
    We have proven, that $L^r$-based strict dissipativity also implies strict distributional dissipativity, which is equivalent to the dissipativity notion based on probability measures for stochastic problems introduced in \cite{Gros2022}. 
    We have seen that this dissipativity notion implies turnpike in the distributions of the optimal solutions. 
    In the non-strict case, $L^r$ dissipativity and distributional dissipativity have turned out to be equivalent.
    
    Some of the open questions that are not answered in this paper are: Under what conditions can one infer strict $L^r$ dissipativity from strict distributional dissipativity? 
    Does the occurrence of the turnpike property in one of its various stochastic forms imply one or the other type of strict dissipativity? 
    Can deterministic dissipativity results be used for the analysis of stochastically perturbed deterministic systems?
    Can one characterize the distributional robustness of the turnpike?

    \JS{Furthermore, we expect that our results can be used to analyze the performance and stability of stochastic MPC schemes. First performance results of this type have been obtained in \cite{PerformanceCDC}. Regarding the stability analysis, the different forms of our dissipativity and turnpike results (in $L^r$, in probability, and in distribution) can be very helpful, since each is directly related to a corresponding stochastic stability property, such as ``stability in the $r$-th mean'', ``stability in probability'', and ``stability in distribution''. Elaborating these connections and using them for the stability analysis of stochastic MPC algorithms is a topic of our current research.} 
   
\section*{References}
    \bibliographystyle{abbrv}
    \bibliography{references}

\end{document}